\documentclass[11pt]{article}
\usepackage{epsfig,psfrag,amsmath,amssymb,amsthm,bbm}
\usepackage[english]{babel}
\usepackage{accents}
\usepackage{enumerate}
\usepackage{amsfonts,amsmath}
\usepackage{bbm}

\usepackage{xcolor}

\newcommand {\IZ}{\mathbb{Z}}
\newcommand {\IN}{\mathbb{N}}  
\newcommand {\IR}{\mathbb{R}}   

\newtheorem{stat}{Statement}

\newtheorem{prop}[stat]{Proposition}
\newtheorem{cor}[stat]{Corollary}
\newtheorem{thm}[stat]{Theorem}
\newtheorem{lemma}[stat]{Lemma}
\newtheorem{remark}[stat]{Remark}

\def\L{{\mathcal L}}

\def\cc{{\mathcal C}}

\def\={&=&}
\def\eps{\varepsilon}

\def\a{\alpha}
\def\b{\beta}

\def\g{\gamma}

\title{ Contact process under renewals I}

\author{Luiz Renato G. Fontes\footnote{Instituto de Matem\'atica e
Estat\'\i stica. Universidade de S\~ao Paulo, SP, Brazil. E-mail:
lrfontes@usp.br},
Domingos H. U. Marchetti\footnote{Instituto de F\'\i sica. Universidade de
S\~ao Paulo, SP, Brazil. Email: marchett@if.usp.br},\\
Thomas S. Mountford \footnote{\'Ecole Polytechnique F\'ed\'erale de Lausanne,
D\'epartement de Math\'ematiques,
1015 Lausanne, Switzerland.
Email: thomas.mountford@epfl.ch}, and
Maria Eulalia Vares\footnote{Instituto de Matem\'atica. Universidade Federal
do Rio de Janeiro, RJ, Brazil. \!Email: eulalia@im.ufrj.br}}

\begin{document}

\maketitle

\begin{abstract}

We investigate a non-Markovian analogue of the Harris contact process in $\mathbb{Z}^d$: an individual is attached to each site $x \in \mathbb{Z}^d$, and it can be infected or healthy; the infection propagates to healthy neighbors just as in the usual contact process, according to independent exponential times with a fixed rate $\lambda$; nevertheless, the possible recovery times for an individual are given by the points of a renewal process with heavy tail; the renewal processes are assumed to be independent for different sites.
We show that the resulting processes have a critical value equal to zero.

\bigskip


\medskip

\end{abstract}


\setcounter{equation}{0}
\section{Introduction}
\label{sec;1}

The classical contact process is a model for (among other phenomena) the spread of an infectious disease.   It was introduced by Harris \cite{H} and has been intensively studied since (see
\cite{L}, \cite{Du1} for overviews).
In this model sites $x \  \in \IZ^{d}$ can be thought of as individuals, $\xi  \ \in \{0,1  \} ^ {\IZ^{d}}$ is a configuration giving the state of health for the population:  $\xi (x) $ represents the state of health of individual $x$,
with  $\xi (x) = 1$ signifying that $x$ is ill and $\xi(x) = 0 $ that individual $x$ is healthy.   In the basic model, sick individuals become healthy at rate $1$ irrespective of the state of health elsewhere while healthy individuals become sick at a rate equal to a parameter $\lambda $ times the number of infected neighbors. The model has many variants.  The rate of infection may not depend only on nearest neighbours or an individual may infect others at a rate depending on the distance between them and so forth.  Equally, the lattice $\IZ^d$ can be replaced by other graphs such as trees or more recently random finite graphs (see $e.g. \ $ \cite{Du2}).  In fact our principal result, Theorem \ref{thm1} holds for the process defined on any infinite connected graph since
such a graph contains a ``copy" of $ \mathbb{N} $. It is usually constructed via a ``Harris system", a collection of Poisson processes.  For the original model one puts Poisson processes of rate $\lambda $ on
ordered pairs $x, y $ of neighbours, which give the times when the individual $x$  ``tries" to infect individual $y$ and one puts a Poisson process of rate $1$ on each site $x$, so that if the individual at $x$ is infected immediately before this time, it gets cured. Motivated by questions regarding long range percolation, we investigate a variant of the contact process in $\mathbb{Z}^d$ in which the death times for the
infection at the sites are determined by independent renewal processes. For brevity we call it renewal contact process (RCP).  Obviously, in no longer using Poisson processes we lose Markovianess. We investigate the possibility of the infection surviving forever when the process starts with a single infected individual, no matter how small $\lambda$ is.

We now become more specific.
\noindent Given a probability measure $\mu$ on $[0, \infty)$ with $\mu\{0\} <1$ and a strictly positive parameter $\lambda$, the RCP admits the following ``Harris graphical construction":

\vspace{0.3cm}

\noindent (I) Let $\{T_{x, i}\}_{x \in \hspace{0.1cm} \IZ^{d},\hspace{0.1cm} i \in \IN}$ be i.i.d. random variables with law $\mu$;

\vspace{0.3cm}

\noindent (II) Let $N=\{N_{x, y}\}_{x \sim y, x,y \in \hspace{0.1cm} \IZ^{d}}$ be a system of i.i.d. rate $\lambda$ Poisson processes, assumed to be independent of $\sigma (\{ T_{x, i}\}_{x \in \IZ^{d}, i \in N})$.
({\bf{Notation:}} $x \sim y$ signifies $\|x-y\|_1=1$, where $\|\cdot\|_1$ stands for the usual $\ell_1$-norm in $\mathbb{R}^d$.)

\vspace{0.3cm}

\noindent For each $x \in \IZ^{d}$ and $n \ge 1$, we write $S_{x,n}=T_{x,1}+\dots + T_{x,n}$, and let $D_x=\{(x,S_{x,n}), n\ge 1\}$, and
use $\mathcal{R}$ to denote any renewal process with interarrival distribution $\mu$.
The segments $\{x\}\times (S_{x,n}, S_{x,n+1})$ will be called \emph{renewal intervals} or \emph{gaps}.

\noindent Given these processes, the RCP is constructed according to the usual recipe:
if $s< t$ and $x,y \in \mathbb{Z}^d$,  a path $\gamma$ from $(x,s)$ to $(y,t)$ is a c\`adl\`ag function on $[s,t]$  for which there exist times  $t_0 = s < t_1 < \cdots < t_k = t$ and  sites $x_0 = x, x_1, \ldots, x_{k-1} = y$ in $\mathbb{Z}^d$ such that $\gamma(u)=x_i$ for $u \in [t_i,t_{i+1})$, and
\begin{itemize}
\item[$\bullet$] $D_{x_i}\cap \{x_i \} \times [t_i,\; t_{i+1}] = \emptyset$ for $i = 0, \ldots, k - 1$;
\item[$\bullet$] $\|x_i-x_{i+1}\|=1$ for $i = 0, \ldots, k-2$;
\item[$\bullet$] $t_i \in N_{x_{i-1}, x_i}$ for $i = 1, \ldots, k-1$.
\end{itemize}
For $A \subset\IZ^{d}$ the RCP $(\xi^{A}_{t})_{t \geq 0}$ starting from initial configuration $\mathbbm{1} _A$ is defined as
$$\xi^{A}_{t}=\{y \colon \textrm{ there exists a path from } (x,0) \textrm{ to } (y,t) \text { for some } x \in A\}.$$

We can identify a configuration $\xi \in \{0,1  \} ^ {\IZ^{d}} $ with the subset (of ${\IZ^{d}} $) $ \{ x: \ \xi(x) = 1\}$.  With this identification we have that (with all processes generated by the same Harris system)
$$
 \xi^A_t \ = \ \cup _ {x \in A } \xi^ { \{x\} }_t.
$$
This property is referred to as {\it additivity}.
We can and will regard the Harris system of Poisson processes and renewal processes as generating simultaneously the processes $ \xi^A_. $  We let $P^{\lambda, \mu} $ to be the probability distribution under which (I) and (II) above hold.  $P^{\lambda, \mu} $ then affixes probabilities to events generated by the processes $\xi^A_. $ generated by these Poisson processes and renewal processes.

\noindent Given a process $(\xi^{A}_{t})_{t \geq 0}$, we write
\begin{equation}
\label{1.1}
\tau^A = \inf \{t: \xi^{A}_{t} = \emptyset \}.
\end{equation}

\noindent Given $\mu$ let $\lambda_c = \inf \{ \lambda : P^{\lambda, \mu} ( \tau^{\{0\}} = \infty) > 0\}$. By translation invariance we have immediately that
$\lambda_c = \inf \{ \lambda : P^{\lambda, \mu} ( \tau^{\{x\}} = \infty) > 0\}$.   From additivity one sees that
for $e.g. \ , \lambda < \lambda_c $ and any finite subset of the integer lattice $A$, that
$P^{\lambda, \mu} ( \tau^{\{A\}} = \infty ) \ = \ P^{\lambda, \mu} ( \max _ {x \in A }\tau^{\{x\}} = \infty)  \ \leq \ \sum_{x \in A } P^{\lambda, \mu} ( \tau^{\{x\}} = \infty) \ = \ 0$.  We thus see that $\{0\}$  may be replaced by any finite set $A$.  By the Hewitt-Savage 0-1 law it follows that whenever $\mu\{0\}=0$, for any $\lambda > \lambda_c$, the renewal ``environment" is a.s. such as to give a strictly positive chance of the $\lambda $ Poisson processes yielding a contact process that never hits configuration $\vec 0$.
\vspace{0.3cm}

\noindent
Our main result is
\begin{thm} \label{thm1}
We make the following assumptions on the interarrival distribution $\mu$:

\noindent A) There exist $1 < M_{1} < \infty$,  $\epsilon_{1} > 0$ and $t_{0} \ \in \ (0, \infty ) $ such that
$$\forall \hspace{0.3cm} t > t_{0}, \hspace{0.2cm} \epsilon_{1} \int_{[0,t]}\,s\, \mu (ds)  < t\mu(t, M_{1} t).$$

\vspace{0.2cm}

\noindent B) There exist $1 < M_{2} < \infty$, $\epsilon_{2} > 0$ and $r_{2} < \infty$ so that $$\forall \hspace{0.3cm} r \geq r_{2}, \hspace{0.3cm}\epsilon_{2} \mu [M_{2}^{r} , M_{2}^{r + 1}] \leq \mu [M_{2}^{r+1}, M_{2}^{r + 2}].$$

\vspace{0.2cm}

\noindent C) There exist $M_3 <\infty$, $\epsilon_{3} > 0$ so that for $t \geq M_{3}$

$$
 t^{- (1 - \epsilon_{3})}\ \leq \ \mu(t,+\infty) \ \leq \ t^{-  \epsilon_{3}} .
$$
\noindent Under these conditions, the critical value for the RCP associated with $\mu$,  $\lambda^ \mu _c$, vanishes.

\end{thm}

\noindent{\bf Comment:} Theorem \ref{thm1} holds without the upper bound assumption in Condition C. This will be included in a 
forthcoming paper. 

In other words, under the above assumptions the random set
\begin{equation}
\label{1.2}
\cc=\{(y,t)\colon \text { there exists a path from } (0,0) \text{ to } (y,t)\},
\end{equation}
is unbounded with positive probability, for all $\lambda>0$.


\noindent {\it Remarks:}

\noindent (a) Given that whenever $d_{1} < d_{2} \in \IN$ we may couple RCPs in dimensions $d_{1}$ and $d_{2}$ in a natural and trivial manner, it suffices to prove Theorem \ref{thm1} for $d=1$.
We will restrict ourselves accordingly.


\noindent (b) In a companion paper \cite{FMV} we consider the situation when the tail of $\mu$ decays as $t^{-\alpha}$ with $\alpha >1$ and, under some rather stringent regularity conditions, we show that
the process on $\IZ$ has strictly positive critical value. Taken together, the two papers identify roughly speaking, tails of order $\frac{L (t)}{t}$ as representing a change of phase where $L(.) $ is a slowly varying function in the one dimensional case.

\noindent (c) A more robust, fairly standard argument, applying to all dimensions, may be given to show that $\lambda_{c} > 0$ if
$\mu$ has finite second moments. See Theorem 2 in \cite{FMV}.

\noindent (d) The critical value equal to $0$ is unusual among contact processes (but see \cite{CD09} and \cite{Du2}). It arises here
as a tunnelling effect where the process traverses larger and larger renewal intervals (with no deaths).
It can also be seen as a type of branching process where the number of offsprings is of infinite mean. But we have not found this mathematically useful, though branching process comparisons are useful in the companion paper \cite{FMV}.

\noindent (e) The proof indeed shows that the same result holds when, besides taking $d=1$, we force the paths to move only in increasing space direction, i.e. only the $N_{x,x+1}, x \in \mathbb{Z}_+$ are considered.

\noindent (f) For conditions A) and B) we can and will suppose that $M_1 \ = \ M_2 $ and $\epsilon _1 \ = \ \epsilon_2 $.  To see this first note that if A) and B) hold for given $M_i $, then
they will also hold for $\epsilon_1 $ and $\epsilon _2 $ replaced by $\epsilon_1 \wedge \epsilon_2 $.  Also
 in condition A) above, if it holds for a value $M_1 $, then it will hold for every $ M > M_1$ with the value $\epsilon_1 $ unchanged.   Thus given A) and B) holding for $(M_1, \epsilon_1)$ and
$(M_2, \epsilon_2)$ respectively where $M_2 \geq M_1 $, then A) and B) both hold for $(M_2, \epsilon_1 \wedge \epsilon_2)$.
If $M_2 < M_1 $, then fix integer $k$ so that $ M_2 ^ k \geq M_1 $.  We claim B) holds for $(M_2^k, \epsilon_2^k / k)$ and we have reduced our case to the previous one. For the claim simply
note that if $j_0 $ is a choice that maximizes $\mu[M_2^{kr+j},M_2^{kr+j+1} ]$ among integers $ 0 \leq j \leq k-1$, then
$$
\mu[M_2^{kr},M_2^{k(r+1)} ] \leq \ k \mu[M_2^{kr+j_0},M_2^{kr+j_0+1} ] \ \leq \ k \epsilon_2 ^{- k} \mu[M_2^{k(r+1)+j_0},M_2^{k(r+1)+j_0+1} ]
$$
$$
\leq  k \epsilon_2 ^{-k} \mu[M_2^{k(r+1)},M_2^{k(r+2)} ].
$$
Accordingly, we can and will suppose that $M_1 $ = $M_2 $ and for simpler reasons that $ \epsilon_1 \ = \ \epsilon_2$.

\noindent (g) Probability laws $\mu$ on $[0,\infty)$ satisfying the conditions of Theorem~\ref{thm1} include all those in the basin of
attraction of an $\a$-stable law, $\a\in(0,1)$. A simple example not in the basin of attraction of a stable law is any one such that
$$
0<\liminf_{t\to\infty}t^\a\mu(t,+\infty)<\limsup_{t\to\infty}t^\a\mu(t,+\infty)<\infty,
$$
with $\a\in(0,1)$. There are also examples with oscillating decay powers. We may take for instance

\begin{equation}\label{ex}
	\mu[t,+\infty)=\exp\left\{-\int_1^t\frac{\eps(s)}s\,ds\right\},\,t>1,
\end{equation}
with $\eps:[1,\infty)\to[\a,\b]$, where $0<\a<\b<1$.
It is straightforward, if tedious, to check that this satisfies A)-C)
and it is not hard to see that $\eps$ may be chosen in such a way that
%
\begin{equation*}
\liminf_{t\to\infty}t^\g \mu(t,+\infty)=0;\quad\limsup_{t\to\infty}t^\g \mu(t,+\infty)=\infty
\end{equation*}
for every $\g\in(\a,\b)$. Indeed, it suffices to define a sequence of numbers $(a_n)_{n\geq0}$ in $[1,\infty)$ increasing sufficiently fast, with $a_0=1$, and make $\eps (s)=\alpha$ if $s \in [a_{2n},a_{2n+1})$, $\eps (s)=\beta$ if $s \in  [a_{2n+1},a_{2n+2})$ for all $n$.



\vspace{0.3cm}

The proof of Theorem~\ref{thm1} largely depends on Lemma \ref{lem1} and Proposition \ref{lem5} proven respectively in Section \ref{sec:2} and in Section~\ref{sec:3}; in the last section we show that the above mentioned tunnelling effect guarantees $\cc$ to be unbounded with positive probability, for any $\lambda>0$.

\section{Reasonable probability of big gaps}
\label{sec:2}

In this and the next section, $(T_i)_{i \ge 1}$ will denote an i.i.d.~sequence having the renewal distribution $\mu$.

\begin{lemma} \label{lem1}
\noindent Under hypotheses A) and B) above, for all $K \ \in \ (0, \infty ) $
$$\inf_{t \geq 1} P ( ( t, Kt) \cap \mathcal{R} = \emptyset ) > 0.$$
\end{lemma}

\vspace{0.3cm}

\begin{proof}
\noindent Under assumption A),  $\mu$ has unbounded support. So it is only necessary to uniformly bound $P ( ( t, Kt ) \cap \mathcal{R} = \emptyset)$ away from zero for large $t$. In particular we may suppose that $t > t_{0}$.

\noindent Let $i_{0} = \inf \{i: T_{i} > t\}$  and consider the following events:

\noindent $A_{1}:= \{T_{i_{0}} \geq Kt\}$

\noindent $A_{2}:\{\sum_{i=1} ^{i_0 -1 }  T_{i}  < t \}$

\noindent $A_{3}:=\{i_{0} \leq  [\frac{\epsilon_{1}}{2 \mu(t, M_{2}t) }]\}$, where $[s]=\max\{k \in \mathbb{Z}, k \leq s\}$. 

\vspace{0.3cm}

We note that on event $A_1 \cap A_2$ we have that $\sum _{i=1} ^ {i_0 -1 } T_i   \ < \ t$ (definition of $A_2$) but $\sum _{i=1} ^ {i_0  } T_i \geq T_{i_0} \ge  Kt $ (event $A_1$)
and so by positivity of the $T_i $, $A_1 \cap A_2 \cap A_3 \ \subset A_1 \cap A_2  \ \subset \ \{  ( t, Kt) \cap \mathcal{R} = \emptyset \}$.  So it suffices to bound $P( A_1 \cap A_2 \cap A_3 )$ from below.
Note that event $A_1$ is independent of the events $A_2 $ and $A_3$. 

Now we have that, as the variables $T_j : \ j \geq \ 1 $ are i.i.d., given that $i_0 \ = \ r $, we have conditionally that the variables $T_i : \ i \leq r $ are independent, the variables $T_i :\ i \ < \ r $ having distribution $T_1 $ conditioned on being less than or equal to $t$, the random variable $T_r $ having distribution $T_1 $ conditioned on having value strictly greater than $t$.
In particular, the probability of event $A_1$ is exactly equal to
$$
\frac{\mu[Kt,+\infty)}{\mu(t,+\infty)}.
$$
As previously noted, we may choose $M_{1} = M_{2}$ without loss of generality and equally suppose that $t=M_1^{k}$ for some $k \ge r_2$
(since we can always increase value $K$).
We pick $r\geq1$ an integer so that $M_{1}^{r} > K$ so  $\mu[Kt,+\infty)\geq \mu[M_{1}^{r} t,+\infty) \geq \frac{\epsilon_{2}^{r} }{r}  \mu (t, M^r_{1}t)$. This follows by the argument given in the discussion of Remark (f) in the preceding section and the assumption that $t$ is an integer power of $M_1$.  Thus
\begin{equation}
\label{delta2}P(A_1) \geq \frac{\epsilon_{2}^{r} }{r}\left(1+ \frac{\epsilon_{2}^{r} }{r}\right)^{-1}=:\delta_2.
\end{equation}
For $P(A_3)$ we have
\begin{eqnarray*}
P(A_3) =1 \ - \ (1-\mu(t,+\infty ) ) ^{ [\frac{\epsilon_1}{2\mu(t, M_2 t)}]}  >  \,   1 \ - \ (1-\mu(t,+\infty) ) ^{ [\frac{\epsilon_1}{2\mu(t, \infty )}]}   \   >\epsilon_1/3 \\
\end{eqnarray*}
if $t_0 $ was fixed sufficiently high, as one easily verifies.
It remains to bound the conditional  probability $P(A_2^c| A_3)$.  Given that the random variables are all non-negative, this is bounded by
$$
\frac{1}{t} E(\sum_{i=1}^{ [\frac{\epsilon_1}{2\mu(t, M_2 t)}]} Y_i)
$$
where the $Y_i$ are i.i.d. random variables equal in distribution to $T_1 $ conditioned on being less than  or equal to $t$.

Using A) we have (again provided $t_0$ had been fixed large enough) that
$ \frac{9}{10}\epsilon_{1} E (Y_1) < t\mu (t, M_{1} t)$, and so by Markov's inequality
$$
P(A_2^c|A_3) \leq    [\frac{\epsilon_1}{2\mu(t, M_2 t)} ]E(Y_1) / t 
\leq 5/9 .
$$

\noindent Now, recalling \eqref{delta2}, for $t$ large, we have
\begin{eqnarray*}
&P (A_{1} \cap A_{2} \cap A_{3}) = \ P(A_3) P(A_2|A_3) P(A_1| A_2 \cap A_3)& \\ &= \ P(A_3) P(A_2|A_3) P(A_1) \geq
\ \frac{\epsilon_1}{3}  \frac{4}{9} \delta_2 \ = \
  \frac{4}{27}\epsilon_1 \delta_2.&
\end{eqnarray*}
\end{proof}
\vspace{0.3cm}
\section{Bounds on renewal sequence missing a far big gap}
\label{sec:3}

We wish to show that for $t$ large any interval in $[t, \infty )$ of length $t^ \epsilon $ for some small $\epsilon$ will be in the
complement of $\mathcal{R} $ outside of probability $t^{- \epsilon }$. It should be noted that even though the $\epsilon $ obtained might be very small indeed, the result will be applied to an exponentially increasing sequence of $t$s and so will yield exponentially decreasing upper bounds.   The strategy is to first show that given such an interval $I$, there will be a larger interval $J$ close by for which the result is true.  We then employ a simple coupling argument to  transfer  to $I$ the bounds for $J$. This will be the main result of this section, stated as Proposition \ref{lem5} below.

\noindent {\it Notation:} Since no confusion arises we use $|A|$ to denote the cardinal of $A$ when $A\subset \mathbb R$ is finite, and also the length for more general Borel sets.

\begin{lemma} \label{lem2}
\noindent Under hypothesis C), there exists $M_3 < \infty$ so that if $I \subset \IR_{+}$ is an interval of length $t \geq M_{3}$, then the probability that $\vert \mathcal{R} \cap I \vert >t^{1 - \epsilon_{3}} \log^{2} t$ is less than $\frac{1}{t}$.
\end{lemma}

\begin{proof}
\noindent By the strong Markov property applied when the renewal process first hits $I$, it is enough to treat the case $I = [0, t].$

\noindent Let $N_{0} = \inf \{n: T_{n} > t\}$. If $N_{0} < m$ then $\vert \mathcal{R} \cap I \vert \leq m$ and
so (increasing $t$ if necessary),
\begin{eqnarray*}
P ( |\mathcal{R} \cap I| > [t^{1 - \epsilon_{3}} \log^{2} t ]) &\leq & P ( N_{0} \geq [t^{1 - \epsilon_{3}} \log^{2} t])\\
&\leq & (1- t^{-(1 - \epsilon_{3})})^{t^{1 - \epsilon_{3}}\log^{2} t -1}\\
&\leq & \frac{1}{1 - t^{-(1 - \epsilon_{3})}} e^{- \log^{2}t } \leq \frac{1}{t},
\end{eqnarray*}
for $t$ large.
\end{proof}

{\it Remark:} For our purposes the bound $\frac{1}{t} $ is somewhat arbitrary and the ``extra" factor $\log^2(t) $ is simply an annoyance.

The next result translates the above bound into the existence of an interval of reasonable size which will, with high probability, be missed by the renewal process.
\begin{cor} \label{cor1}
\noindent There exists a finite constant $t_{1}$ such that for all $t \geq t_{1}$, for all interval  $I$ of length $t$ in $\IR_{+}$ there exists an interval $J \subset I $ of length $t^{\frac{\epsilon_{3}}{2}}$ so that
$$P( \mathcal{R} \cap J \not= \emptyset ) \leq t^{-\frac{\epsilon_{3}}{3}}.$$
\end{cor}

\vspace{0.3cm}

\begin{proof}
For $u \in \IR $ and nonempty $A \ \subset \ \IR $, we write $d(u,A) \ = \ \inf \{ |u-y|: y \ \in \ A \} .$
\noindent Let $g(u) = P [d(u, \mathcal{R}) >\frac{1}{2}t^{\frac{\epsilon_{3}}{2}} ], \hspace{0.3cm} u \in I.$

\noindent If  $g(u) > 1- t^{-\frac{\epsilon_{3}}{3}}$ for some $ u \in I$, then our result follows by taking
$J=[ u-\frac12t^{\frac{\epsilon_{3}}{2}}, u+\frac12t^{\frac{\epsilon_{3}}{2}} ]$. Thus, let us suppose no such $u$ exists. Then, 
\begin{equation*}
E ( \vert (\mathcal{R} + [ - \frac{1}{2} t^{\frac{\epsilon_{3}}{2}}, \frac12t^{\frac{\epsilon_{3}}{2}}]) \cap I \vert ) \geq t^{1 - \frac{\epsilon_{3}}{3}}.
\end{equation*}
But by the preceding Lemma, taking $t_1$ large we have

\begin{equation*}
E ( \vert (\mathcal{R} + [ - \frac{1}{2}t^{\frac{\epsilon_{3}}{2}}, \frac{1}{2}t^{\frac{\epsilon_{3}}{2}}]) \cap I \vert ] \leq 1 + ( t^{1 -\epsilon_{3}} \log^{2} t)  t^{\frac{\epsilon_{3}}{2}} < t^{1 - \frac{ \epsilon_{3}}{3}}\hspace{0.2cm}  \textrm{for} \hspace{0.2cm} t\geq t_{1}.
\end{equation*}
\end{proof}

Having established that for an interval far away from $0$, there will be a large interval which is ``missed" by the renewal sequence with high probability, we wish to
use a simple coupling argument to show that this property must hold for all intervals ``reasonably close" to this interval.  The coupling between renewal processes that we use is really a coupling between random walks as explained in $e.g. \  \cite{O}$ or $ \cite{denH} $.

\noindent Given the law of the $\{T_{i}\}_{i\ge 1}$ there exists a bounded interval $I$ such that $P ( T_{1} \in I ) > 0$ and $P ( T_{1} = T_{2} \vert  T_{1}, T_{2} \in I ) < 1.$

\noindent Given $V_{0} > 0$, the $(V_{0})-$coupling between two identically distributed renewal sequences (sharing the above renewal time distribution), $\{T_{i}\}_{i \geq 1}$ and $\{\tilde T_{i}\}_{i \geq 1}$ (here the $T_{i}, \tilde T_{j}$ are ``interarrival times" and so identically distributed) is as follows:

\vspace{0.3cm}
\noindent Let $ N_{V_{0}} = \inf \{k : \sum_{i=1}^{k}(T_{i} - \tilde T_{i}) > V_{0}\}.$

\noindent For $i > N_{V_{0}} \hspace{0.3cm} T_{i} = \tilde T_{i}$, independent of preceding realizations

\noindent For $i \leq N_{V_{0}}$ we choose $(T_{i}, \tilde T_{i})$ to be independent of preceding realizations and with the property that
\vspace{0.3cm}

a) $T_{i} \in I \Longleftrightarrow \tilde T_{i} \in I$;

b) $T_{i} \in I^{c} \Rightarrow \tilde T_{i} = T_{i}$;

c) given $\{T_{i} \in I\} = \{\tilde T_{i} \in I\} = \{T_{i} , \ \tilde T_{i} \in I\}$  the variables $T_{i}$ and $\tilde T_{i}$ are i.i.d. with common distribution equal to that of $T_{1}$ conditioned on $\{T_{1} \in I\}$.

\vspace{0.3cm}

\noindent It is immediate that $(\sum_{i=1}^{n}(T_{i} - \tilde T_{i}))_{n \geq 0}$ is equal in law to a random walk with distribution $(T_{1} - \tilde T_{1})$ stopped at the hitting time for $(V_{0}, +\infty).$

\vspace{0.3cm}

\noindent Given that the law of $T_{1} - \tilde T_{1}$ is symmetric and non trivial, so that $E(T_{1} - \tilde T_{1}) = 0$ and $ E ((T_{1} - \tilde T_{1})^{2} ) < \infty$:


\begin{lemma} \label{cor2}
\noindent There is a constant $K$, depending on $I$ and the distribution of $T_{1}$ conditional on being in $I$, so that for all $t > 0 $ and $V_0 \geq 1$
$$P ( N_{V_{0}} > t) \leq \frac{K V_{0}}{\sqrt{t}}.$$
\end{lemma}

\begin{proof}

We fix $M < \infty $ so that $I \ \subset \ [0,M]$, so that a.s. $ \forall i \ T_i - \tilde T_{i} \subset [-M,M]$.

We analyze the random walk $(\sum_{i=1}^{n}(T_{i} - \tilde T_{i}))_{n \geq 0}$ by the standard Brownian embedding for symmetric random variables.  Given standard Brownian motion
$(B_s)_{s \geq 0}$, independent of the $T_i, \ \tilde T_{i} \ i \geq 1 $ random variables,we define stopping times (for the natural filtration of $B_.$ augmented by the  $T_i, \ \tilde T_{i} \ i \geq 1 $)
$$
S_0 \ = \ 0; \quad S_{i+1} \ = \ \inf \{t> S_i: |B_t-B_{S_i}| \ = \ |T_{i+1}- \tilde T_{i+1}|\}.
$$
Then $(B_{S_i})_ {i \geq 1} $ is equal in law to our random walk and the variables $S_i - S_{i-1} : \ i \geq 1 $ are i.i.d. random variables having all moments finite.

We define the stopping time
$$
 \sigma \  =  \ \inf \{ s: B_s > V_0 +3M \}.
$$
We now note that if $k = \max \{j: S_j \leq \sigma\} $, then $S_k > V_0 $.
 Then (without loss of generality) taking $t$ to be an integer, $P( N_{V_0} > t)$ is less than or equal to
$$
P(S_t > 2t E(S_1) ) \ + \ P( \sup _{s \leq 2t E(S_1)} B_s \ < V_0 + 3M).
$$
Given the existence of all moments of $S_1$, the first term decays to zero faster than any power of $t$ so the result follows from Brownian hitting probabilities.
\end{proof}
\vspace{0.3cm}

\begin{prop} \label{prop6}
There exists $K < \infty$ so that for all $ n$ and all $ V_0 \leq 2^{n}$ there is a coupling between renewal sequences $\{T_{i}\} , \{\tilde T_{i}\}$ such that $\sum^{N_{V_0}}_{i = 1} T_{i}$ and $\sum^{N_{V_0}}_{i = 1} \tilde T_{i}$ are both less than $2^{Kn}$ outside a set of probability $C 2^{-n}$ for $C$ universal and $n$ large.
\end{prop}

\vspace{0.3cm}

\begin{proof}
\noindent Since, by Lemma \ref{cor2}, $P( N_{V_0} \geq 2^{4n}) \leq K  2^{-n}$, it remains to show that  $P ( \sum_{i=1}^{2^{4n}} T_{i} \geq 2^{Kn} ) \leq 2^{-n}$ for $n$ large. But this follows easily from hypothesis C): $P(\sum_{i=1}^{2^{4n}} T_{i} \geq 2^{Kn} ) \leq 2^{4n} P (T_1 \geq 2^{(K-4)n}) \ \leq \ 2^{4n}  2^{-(K-4)n \epsilon_3}$ which is less than $2^{-n}$ for all $n$ large if $K $ was fixed large enough.
\end{proof}

\vspace{0.3cm}
We are now ready to prove the main result of this section.  More refined estimates (see $e.g. \ \cite{E}$ and more recently $ \cite{Chi}$ or $ \cite{CD}$) are available but they require greater regularity on the tails of the distribution of $T_1$.

\begin{prop} \label{lem5}
There exists $ \epsilon_{4} > 0$ so that for $n$ large and for all intervals  $[s, t]$ with $2^{n}\leq s \leq t \leq s +2^{n \epsilon_{4}}$, one has
\begin{equation}
\label{gap2}
P ( \mathcal{R} \cap [s,t] \not= \emptyset ) \leq 2^{-n \epsilon_{4}}.
\end{equation}
\end{prop}

\vspace{0.3cm}

\begin{proof}
\noindent
We take $\epsilon_{4} < \frac{\epsilon^{2}_{3}}{27 K^{2}}$, where $\epsilon_{3}$ is a constant as in hypothesis C), and $K$ is as in Proposition \ref{prop6} (absorbing the constant $C$ therein as well).
For $s \ge 2^n$, consider the interval $[s, s+ 2^{\frac{4 n \epsilon_{4}}{\epsilon_{3}}}]$, to which we apply Corollary \ref{cor1}, guaranteeing the existence of a subinterval $J $ of length $2^{2n \epsilon_{4}}$ so that
\begin{equation}
\label{gap}
P (\mathcal{R} \cap J \not= \emptyset ) \leq 2^{-\frac{4}{3}n \epsilon_{4}}.
\end{equation}

\vspace{0.3cm}

\noindent We now choose $V_{0}$ so that $[s + V_{0}, t + V_{0} ]$ is within $J$ and at distance at least $\frac{2^{n \epsilon_{4}}}{3}$ from $J^{c}$; note that $\frac{2^{n \epsilon_{4}}}{3} \leq V_{0} \leq 2^{4 n \frac{\epsilon_{4}}{\epsilon_{3}}}$.

\vspace{0.3cm}

\noindent We couple together renewal sequences $\{T_{i}\}, \{\tilde T_{i}\}$ so that outside a set of probability $K 2^{-2 n \frac{\epsilon_{4}}{\epsilon_{3}}}$ one has:

$$N_{V_{0}} \leq 2^{12 n \frac{\epsilon_{4}}{\epsilon_{3}}}$$
and
$$\max\{\sum_{i=1}^{N_{V_{0}}}T_{i}, \sum_{i=1}^{N_{V_{0}}}\tilde T_{i}\}\leq 2^{4 n K \frac{\epsilon_{4}}{\epsilon_{3}}} < 2^{n}.$$

\noindent Thus we have (where the superscript to $\mathcal{R}$ indicates the corresponding interarrival sequence):
$$P ( \mathcal{R}^{\tilde T} \cap [s, t] \not= \emptyset ) \leq P ( \mathcal{R}^{T} \cap J \not= \emptyset ) + K 2^{-2 n \frac{\epsilon_{4}}{\epsilon_{3}}}.$$

By \eqref{gap}, we have then that $ P ( \mathcal{R}^{\tilde T} \cap [s, t] \not= \emptyset ) \leq 2^{-\frac{4}{3}n \epsilon_{4}}+ K 2^{-2 n \frac{\epsilon_{4}}{\epsilon_{3}}}. $
For all $n$ large we must have that this latter bound is below $2^{-n \epsilon_{4}}$, which was precisely the desired bound \eqref{gap2}.

\end{proof}

\section{Proof of the theorem}
\label{proof-thm}

In proving that the critical value is equal to $0$, we have to show that for any $\lambda > 0$ there is a strictly positive chance that the contact process on $\IZ$ starting from a single ``infected" site survives for all time.  In fact we really only consider the contact process on $\IN $.  In turn our argument becomes  that, loosely speaking, having infected a ``large" interval (depending on $\lambda$) survival becomes probable.
\vspace{0.3cm}

\noindent We can assume that for $t_{0}$ to be chosen later, the renewal process for site $0$ has $T_{0,1} > 4 t_{0}$. We also fix $\gamma=\epsilon_4>0$ given by Proposition \ref{lem5}.

\noindent We define recursively levels $L_0, L_{1}, L_{2}, \dots$: $L_{0}=0$ and for $i \ge 1$,
\begin{equation}
\label{level}
L_{i} = \inf \{ k > L_{i-1} : D_{k} \cap \{ k \} \times [ t_{0} 2^{i} , t_{0} 2 ^{i + 2} ] = \emptyset \}.
\end{equation}

\begin{remark}
 \label{markov}
 Note that given $\mathcal{F}_{L_{i}}:=\sigma ( L_i, D_{k} : k \leq L_{i} )$ the renewal processes at sites $(L_{i} + j)_{j \ge 1}$ are independent and identically distributed.
 \end{remark}

\noindent For $i \ge 1$, we now define ``bad" events $B_{i}$ to be the union of the following events:

(I) $\{L_{i} > L_{i-1} + i \log(t_0)\}$.

(II) $\{\exists k\in\{ L_{i - 1} + 1, \dots, L_{i}\}$ so that $D_{k} \cap \{ k \} \times [ t_{0} 2^{i} - (t_{0} 2^{i} )^{\gamma}, t_{0} 2^{i}] \not= \emptyset\}.$

(III) $\{\exists k\in\{ L_{i - 1}, \dots, L_{i}-1\}$ so that there are no marks of $N_{k,k+1}$ in the time interval $( t_{0} 2^{i} - (t_{0} 2^{i} )^{\gamma}+ \frac{k-L_{i-1}}{i\log t_0} (t_{0} 2^{i})^{\gamma}, t_{0} 2^{i}- (t_{0} 2^{i})^{\gamma} + \frac{k+1-L_{i-1}}{i\log t_0} (t_{0} 2^{i})^{\gamma})\}$.

\vspace{0.3cm}

\begin{lemma}\label{lem 4}
We can fix $t_0$ large enough so that $P(B_{i}) \leq K \exp (-ci) \hspace{0.3cm} \forall \hspace{0.2cm} i$, for some $c > 0, K < \infty$.
Furthermore, if $t_0$ is taken large enough, we will have $\sum_{i=1} ^ \infty P(B_i ) < 1/2$.
\end{lemma}

\vspace{0.3cm}
\begin{proof}
\noindent By Lemma \ref{lem1} and Remark \ref{markov}, the events
$$V_{j} = D_{L_{i-1} + j} \cap \{ L_{i-1} + j  \} \times [t_{0} 2^{i}, t_{0} 2^{i + 2}] = \emptyset$$
\noindent are independent and independent of $\mathcal{F}_{L_{i-1}}$ having probability $c_{1} > 0$, provided $t_{0}$ has been fixed sufficiently large. Hence
$$P(L_{i} > L_{i-1}+ i\log t_0) \leq (1 - c_{1})^{i\log(t_0)}$$

\vspace{0.1cm}

\noindent By Proposition \ref{lem5} the probability of (II) occurring and $L_{i} \leq L_{i-1}+ i\log t_0$ is bounded by $i\log(t_0) (t_{0} 2^{i})^{- \gamma}$ (again supposing $t_{0}$ is large). Similarly the intersection of (III) and $L_{i} \leq L_{i-1}+ i\log t_0$ has a probability bounded by
$$i \log(t_0)e^{- \lambda(t_{0} 2^{i})^{\gamma} /i\log(t_{0})}$$
\end{proof}

\noindent {\it Proof of Theorem \ref{thm1}} \\
Let $ \lambda > 0 $ be any strictly positive value.
We choose $t_0 $ so large that Lemma \ref{lem 4} holds and in particular that $\sum_{i=1} ^ \infty P(B_i ) < 1/2$.  Then we simply observe that on the intersection of $\{ T_{0,1} > 4t_0\}$ and $\cap _{i \geq 1}B_i^ c$,
the RCP starting with a single infected site at $0$ survives forever.  Thus the survival probability is strictly positive.  Given that $\lambda $ can be as small as desired the result is proven. \qed

\bigskip

\noindent {\textbf {Acknowledgements:}} \\L.R.G. Fontes acknowledges support of CNPq (grant 311257/2014-3)
and Fapesp (grant 2017/10555-0). M.E. Vares acknowledges support of CNPq (grant 305075/2016-0) and FAPERJ (grant E-26/203.048/2016).

\end{document}